\documentclass[a4paper,11pt]{article}

\usepackage[leqno]{amsmath}
\usepackage{amssymb,amsthm,ifthen}
\usepackage{url,enumitem}
\usepackage[pdftex]{graphicx}
\usepackage[curve,knot]{xypic}

\usepackage{yhmath}

\usepackage{tikz}
\usetikzlibrary{calc}
\usepgflibrary{shapes.geometric}
\usepgflibrary{shapes.misc}
\usetikzlibrary{positioning}
\usetikzlibrary{decorations}
\usetikzlibrary{arrows}

\usepackage[a4paper,left=1in,width=458pt,textheight=650pt,top=1.5in]{geometry}

\usepackage{bbm}

\usepackage{hyperref}

\input cyracc.def

\hypersetup{
pdftitle={Graded cluster algebras},
pdfauthor={Jan E. Grabowski},
pdfstartview=FitH
}

\newcommand{\bigdsum}{\bigoplus}

\newcommand{\card}[1]{\lvert #1 \rvert}

\let\chisave\chi
\renewcommand{\chi}{{%
 \mathchoice{\raisebox{0.25ex}{$\displaystyle\chisave$}}
            {\raisebox{0.2ex}{$\textstyle\chisave$}}
            {\raisebox{0.2ex}{$\scriptstyle\chisave$}}
            {\raisebox{0.1ex}{$\scriptscriptstyle\chisave$}}}}

\newcommand{\complex}{\ensuremath \mathbb{C}}

\newcommand{\cross}{\times}
\newcommand{\curly}[1]{\ensuremath{\mathcal{#1}}}

\newcommand{\defeq}{\stackrel{\scriptscriptstyle{\mathrm{def}}}{=}}
\let\degsave\deg
\renewcommand{\deg}{\underline{\degsave}}

\newcommand{\dimvec}{\ensuremath \mbox{\underline{dim}}}

\newcommand{\dsum}{\ensuremath{ \oplus}}

\newcommand{\End}[2]{\ensuremath \mbox{End}_{#1}(#2)}
\renewcommand{\epsilon}{\varepsilon}

\newcommand{\ind}[2]{\text{\underline{ind}}_{#1}(#2)}

\newcommand{\integ}{\ensuremath{\mathbb{Z}}}

\newcommand{\iso}{\ensuremath \cong}

\newcommand{\op}[1]{{#1}^{\mbox{\scriptsize \textup{op}}}}

\renewcommand{\phi}{\varphi}

\newcommand{\real}{\ensuremath \mathbb{R}}

\newcommand{\transpose}[1]{{#1}^{T}}

\renewcommand{\leq}{\leqslant}
\renewcommand{\geq}{\geqslant}

\theoremstyle{plain}
\newtheorem{theorem}{Theorem}[section]
\newtheorem*{theorem*}{Theorem}
\newtheorem{proposition}[theorem]{Proposition}
\newtheorem{lemma}[theorem]{Lemma}
\newtheorem{corollary}[theorem]{Corollary}

\theoremstyle{definition}
\newtheorem{definition}[theorem]{Definition}
\theoremstyle{remark}
\newtheorem{remark}[theorem]{Remark}
\newtheorem{example}[theorem]{Example}
\newtheorem*{example*}{Example}
\newtheorem*{examplectd*}{Example (continued)}

\title{Graded cluster algebras}
\author{Jan E. Grabowski\footnotemark[2] 
\\ \small{\textit{Department of Mathematics and Statistics, Lancaster University,}}
\\ \small{\textit{Lancaster, LA1 4YF, United Kingdom}}
}
\date{18th June 2015}

\begin{document}

\maketitle

\renewcommand{\thefootnote}{\fnsymbol{footnote}}
\footnotetext[2]{Email: \url{j.grabowski@lancaster.ac.uk}.  Website: \url{http://www.maths.lancs.ac.uk/~grabowsj/}}
\renewcommand{\thefootnote}{\arabic{footnote}}
\setcounter{footnote}{0}

\begin{abstract} In the cluster algebra literature, the notion of a graded cluster algebra has been implicit since the origin of the subject.  In this work, we wish to bring this aspect of cluster algebra theory to the foreground and promote its study.

We transfer a definition of Gekhtman, Shapiro and Vainshtein to the algebraic setting, yielding the notion of a multi-graded cluster algebra.  We then study gradings for finite type cluster algebras without coefficients, giving a full classification.

Translating the definition suitably again, we obtain a notion of multi-grading for (generalised) cluster categories.  This setting allows us to prove additional properties of graded cluster algebras in a wider range of cases.  We also obtain interesting combinatorics---namely tropical frieze patterns---on the Auslander--Reiten quivers of the categories.

\vspace{1em}

\end{abstract}

\tableofcontents

\vfill

\pagebreak

\section{Introduction}

Gradings for cluster algebras have been introduced in various ways by a number of authors and for a number of purposes.  The evolution of the notion started with the foundational work of Fomin and Zelevinsky (\cite{FZ-CA1}), who consider $\integ^{n}$-gradings where $n$ is precisely the rank of the cluster algebra.  

Gekhtman, Shapiro and Vainshtein have also given a definition of a multi-graded cluster algebra more generally, in the dual language of toric actions (\cite[Section~5.2]{GSV-Book}).  In the case that the underlying field $\mathbb{F}$ is $\real$ or $\complex$, they discuss a toric action of $(\mathbb{F}^{\ast})^{r}$ determined by a choice of integer weight.  Their Lemma~5.3 states a necessary and sufficient condition that a local toric action on a seed extends to a global one on the associated cluster algebra.

Berenstein and Zelevinsky (\cite[Definition~6.5]{BZ-QCA}) have given a definition of graded quantum seeds, which give rise to module gradings but not algebra gradings.  Then in work with St\'{e}phane Launois (\cite{GradedQCAs}), in which we proved that the quantum versions of homogeneous coordinate rings of Grassmannians are quantum cluster algebras, we independently introduced the notion of a $\integ$-grading for a quantum cluster algebra and noted that the definition also applies to the classical commutative case.

In the first part of this note, we briefly set out the definition of a $\integ^{d}$-graded seed and cluster algebra, in the algebraic setting.  Since the proofs reduce to the $\integ$-graded case, given in detail in this language in \cite{GradedQCAs}, we omit these.  These definitions and results are equivalent to those of Gekhtman--Shapiro--Vainshtein, though the toric action setting suggests a different set of associated questions and we recommend Chapter~5 of \cite{GSV-Book} to readers interested in the more geometrical aspects.  We will concentrate only on algebraic and combinatorial aspects here.

We wish to promote the use of gradings in cluster algebra theory and to show that there are interesting questions and especially combinatorial phenomena associated with gradings.  To do this, we consider the usual starting case of finite type cluster algebras without coefficients.  In this case, we can give a complete classification of the gradings that occur.  In particular we observe that the gradings we obtain are all \emph{balanced}, that is, there is a bijection between the set of variables of degree $\underline{d}$ and those of degree $-\underline{d}$.

Next we introduce the notion of a graded (generalised) cluster category.  The idea of the definition is the same as previously: to associate an integer vector (the multi-degree) to an object in the category in such a way that the vectors are additive on distinguished triangles and transform naturally under mutation.  This is done via the key fact that every object in a generalised cluster category has a well-defined index; in order to satisfy the aforementioned two properties, degrees are necessarily linear functions of the index.

In finite types, indices are very closely related to both dimension vectors and almost positive roots---it is not surprising that the combinatorics of cluster algebra gradings should be closely related to the latter.  However the theory of generalised cluster categories goes very far beyond finite type.

The categorical approach has the advantage that it encapsulates the global cluster combinatorics, or more accurately the set of indices does.  Another consequence is an explanation for the observed balanced gradings in finite type: we show that the auto-equivalence of the cluster category given by the shift functor induces an automorphism of the set of cluster variables that reverses signs of degrees.  Hence any cluster algebra admitting a cluster categorification necessarily has all its gradings being balanced, for example finite type or, more generally, acyclic cluster algebras having no coefficients.

Two of the highlights of the resulting analysis are firstly the close link between the gradings on the cluster algebra and the representation theory encoded in the associated cluster category, and secondly the emergence of a combinatorial pattern called a tropical frieze on the Auslander--Reiten quiver of the cluster category, arising from the degrees of the cluster variables.  For us, this illustrates how deeply integrated into the theory of cluster algebras gradings are.

The structure of this paper is as follows.  We begin with a brief exposition of the definition of a (multi-)grading for a cluster algebra and some associated lemmas (Section~\ref{s:gradedCAs}).  We then classify gradings for coefficient-free cluster algebras of finite type (Section~\ref{s:gradings-in-finite-type}), using a result of Fomin and Zelevinsky on Laurent expressions for cluster variables in finite type.  (This encompasses all finite type cases, not just the simply laced ones.)

We then turn to cluster categories and show that we can introduce the multi-gradings at the categorical level (Section~\ref{s:graded-cluster-categories}).  As a consequence of the graded cluster category setting, we obtain so-called tropical friezes on these categories and, when the generalized cluster category comes from a derived category, on that derived category too.  This is explained and illustrated in Section~\ref{s:tropical-friezes}.

We conclude with some remarks in Section~\ref{s:homogenisation} on how one may add coefficients to an initial seed that does not admit a grading, so that the new seed does---that is, how to homogenise a cluster algebra.

\subsection*{Acknowledgements}

The author would particularly like to thank St\'{e}phane Launois for many helpful discussions throughout the collaboration that this work originated from, Robert Marsh for useful references regarding cluster categories and Thomas Booker-Price for permission to reproduce the results of his calculations in Sections~\ref{ss:gradings-in-type-B} and \ref{ss:gradings-in-type-C}.  We also thank numerous colleagues with whom preliminary versions of these results were discussed following their presentation at various institutions.

\section{Preliminaries}

The construction of a cluster algebra of geometric type from an initial seed $(\underline{x},B)$, due to Fomin and Zelevinsky (\cite{FZ-CA1}), is now well-known.  Here $\underline{x}$ is a transcendence base for a certain field of fractions of a polynomial ring and $B$ is a skew-symmetrizable integer matrix; in the skew-symmetric case $B$ is often replaced by the quiver $Q=Q(B)$ it defines in the natural way.  For simplicity, we consider our base field to be $\mathbb{Q}$.

We refer the reader who is unfamiliar with this construction to the survey of Keller (\cite{Keller-CASurvey}) and the book of Gekhtman, Shapiro and Vainshtein (\cite{GSV-Book}) for an introduction to the topic and summaries of the main related results in this area.

We set some notation for later use.  For $k$ a mutable index, set
\begin{align*}
\underline{b}_{k}^{+} & = -\underline{\boldsymbol{e}}_{k}+\sum_{b_{ik}>0}b_{ik}\underline{\boldsymbol{e}}_{i} \qquad \text{and} \\
\underline{b}_{k}^{-} & = -\underline{\boldsymbol{e}}_{k}-\sum_{b_{ik}<0}b_{ik}\underline{\boldsymbol{e}}_{i}
\end{align*}
where the vector $\underline{\boldsymbol{e}}_{i}\in \integ^{r}$ ($r$ being the number of rows of $B$) is the $i$th standard basis vector.  Note that the $k$th row of $B$ may be recovered as $B_{k}=\underline{b}_{k}^{+}-\underline{b}_{k}^{-}$.

Then given a cluster $\underline{x}=(X_{1},\ldots,X_{r})$ and exchange matrix $B$, the exchange relation for mutation in the direction $k$ is given by
\[ X_{k}^{\prime}=\underline{x}^{\underline{b}_{k}^{+}}+\underline{x}^{\underline{b}_{k}^{-}} \]
where for $\underline{a}=(a_{1},\dotsc ,a_{r})$ we set
\[ \underline{x}^{\underline{a}} = \prod_{i=1}^{r} X_{i}^{a_{i}}. \]

Later we will briefly discuss cluster algebras with coefficients (also called frozen variables).  That is, we designate some of the elements of the initial cluster to be mutable (i.e.\ we are allowed to mutate these) and the remainder to be non-mutable.  We will also talk about the corresponding indices for the variables as being mutable or not; in \cite{BZ-QCA} the former are referred to as exchangeable indices.  The rank of the cluster algebra is the number of \emph{mutable} variables in a cluster; we will refer to the total number of variables, mutable and not, as the cardinality of the cluster.

We will retain the usual convention that $B$ will be a matrix with rows indexed by the initial cluster variables and columns indexed by the mutable initial cluster variables.  The matrix $B_{\text{mut}}$ obtained by taking only the rows of $B$ corresponding to mutable variables is the principal part of $B$.


\section{Multi-graded seeds and cluster algebras}\label{s:gradedCAs}

The natural definition for a multi-graded seed is as follows.

\begin{definition} A multi-graded seed is a triple $(\underline{x},B,G)$ such that
\begin{enumerate}[label=(\alph*)]
\item $(\underline{x}=(X_{1},\dotsc ,X_{r}),B)$ is a seed of cardinality $r$ and
\item $G$ is an $r\cross d$ integer matrix such that $B^{T}G=0$.
\end{enumerate}
\end{definition}

From now on, we use the term ``graded'' to encompass multi-graded; if the context is unclear, we will say $\integ^{d}$-graded.  

The above data defines $\deg_{G}(X_{i})=G_{i}$ (the $i$th row of $G$) and this can be extended to rational expressions in the generators $X_{i}$ in the obvious way.  We also need to be able to mutate our grading, which we do via the matrix $E$ (denoted $E_{+}$ in \cite{BZ-QCA}) that encodes mutation of $B$:
\[ E_{rs}=
\begin{cases} 
\delta_{rs} & \text{if}\ s\neq k; \\
-1 & \text{if}\ r=s=k; \\
\max(0,-b_{rk}) & \text{if}\ r\neq s=k.
\end{cases}
\]
Then we have that $B^{\prime}=EB\transpose{E}$.  Setting $G^{\prime}=\transpose{E}G$, it is straightforward to verify that the $i$th row of $G^{\prime}$ is given by
\[ G_{i}^{\prime} = \begin{cases} G_{i} & \text{if}\ i\neq k \\ (\underline{b}_{k}^{-})^{T}G & \text{if}\ i=k \end{cases}. \]
and furthermore $(B^{\prime})^{T}G^{\prime}=0$ so that $(\underline{x}^{\prime},B^{\prime},G^{\prime})$ is again a graded seed.

Note that we have that if $Y\in \underline{x}'$, $\deg_{G}(Y)=\deg_{G'}(Y)$ by definition: the degree of $Y$ with respect to $G$ is precisely the $k$th row of $G'$, which is also the degree of $Y$ viewed as an element of the graded seed $(\underline{x}',B',G')$.

Then we see that repeated mutation propagates a grading on an initial seed to every cluster variable and hence to the associated cluster algebra, as every exchange relation is homogeneous.

\begin{proposition}\label{p:gradedCA} The cluster algebra $\curly{A}(\underline{x},B,G)$ associated to an initial graded seed $(\underline{x},B,G)$, with $G$ an $r\cross d$ integer matrix, is a $\integ^{d}$-graded algebra. Every cluster variable of $\curly{A}(\underline{x},B,G)$ is homogeneous with respect to this grading.  \qed
\end{proposition}

\begin{remark} If we have a grading $G$, each column of $G$ is itself a cluster algebra grading.  Similarly, every $J\subseteq \{1,\dotsc ,d\}$ gives rise to a $\integ^{\card{J}}$-grading, by taking the submatrix of $G$ on the column subset $J$.

Also notice that $d$ is independent of $r$: every cluster algebra admits a $\integ^{d}$-grading for any $d\geq 1$, namely taking $G$ to be the $r\cross d$ zero matrix.  Equally, if $H\in \integ^{r}$ defines a non-zero $\integ$-grading, so does $H^{(d)}$, the matrix with $d$ columns all equal to $H$, again for any $d$.
\end{remark}

\begin{remark} From the definition of a grading, we see that the existence of a grading is a linear algebra problem: if $B$ has rank equal to the number of mutable indices\footnote{That is, if the (row) rank of the matrix $B$ equals the rank of the cluster algebra---an unfortunate coincidence of terminology.}, the only solution is the zero grading $\underline{0}$, assigning degree $0$ to every cluster variable.

Classification of gradings for a particular $B$ is also a linear algebra problem, of finding a basis for the kernel in the case that the rank is not maximal.  (Here, and below, ``kernel'' refers to the kernel of the map of free abelian groups induced by $B$.  We will also say e.g.\ $\dim \ker B$ rather than the more usual group-theoretic term ``rank'', to avoid further overloading that word.)

However it will in general be difficult to find the degrees of \emph{every} cluster variable, especially in infinite types.  In finite types, one can reasonably expect to solve this problem and we will do so in two ways, one algebraic and one categorical.
\end{remark}

\begin{remark} For some cluster algebra problems, the presence or absence of coefficients does not play a large part and the phenomena seen are determined by the cluster algebra type.  This is not the case for gradings, though.  As the examples later will illustrate, adding or removing coefficients can radically change the gradings that can exist.  This is to be expected: adding coefficients increases the number of rows of the associated exchange matrix and this can impact on the rank and hence the solutions $G$ that are the grading vectors.
\end{remark}

We conclude this section by recording some elementary results on a particular class of gradings which, as we see, essentially contain information about every possible grading.

\begin{definition}  Let $(\underline{x},B)$ be a seed.  We call a multi-grading $G$ whose columns are a basis for the kernel of $B$ a \emph{standard} multi-grading, and call $(\underline{x},B,G)$ a standard graded seed.
\end{definition}

\begin{lemma}\label{l:mut-of-std-is-std} Let $\Sigma=(\underline{x},B,G)$ be a standard graded seed.  Then any mutation of $\Sigma$, say $\Sigma'=(\underline{x}',B',G')$, is again a standard graded seed.  Hence any graded seed that is mutation equivalent to a standard graded seed is itself standard.
\end{lemma}

\begin{proof} Recall that we have $G'=E^{T}G$ and $(B')^{T}G'=0$.  So the columns of $G'$ certainly belong to $\ker B'$ and since $E$ is invertible, the column rank of $G'$ is equal to that of $G$.  (Indeed $E^{2}$ is the identity, corresponding to matrix mutation being an involution.)  Noting that mutation also preserves rank, since $B'=EBE^{T}$, so that $\dim \ker B=\dim \ker B'$, the columns of $G'$ form a basis for $\ker B'$ as required.

The final claim is immediate.
\end{proof} 

\begin{lemma}\label{l:chg-of-basis-grading} Let $(\underline{x},B,G)$ be a standard graded seed and let $H$ be any grading for $(\underline{x},B)$.  Then there exists an integer matrix $M=M(G,H)$ such that for any cluster variable $Y$ in $\curly{A}(\underline{x},B,H)$ we have 
\[ \deg_{H}(Y)=\deg_{G}(Y)M \]
where on the right-hand side we regard $Y$ as a cluster variable of $\curly{A}(\underline{x},B,G)$ in the obvious way.
\end{lemma}

\begin{proof} Since $G$ is standard, its columns are a basis for the kernel of $B$.  Furthermore every column of $H$ belongs to $\ker B$ since $H$ is a grading, and so there exists a matrix $M$ encoding the expression of the columns of $H$ in the basis of columns of $G$, i.e. $H=GM$.  Hence if $Y$ is an initial cluster variable, i.e. $Y\in \underline{x}$, we have the result.

It then suffices to show that the result remains true under mutation, and the full statement will follow by induction.  Let $X_{k}'=\mu_{k}(\underline{x},B)$ be the mutation of the seed $(\underline{x},B)$ in the direction $k$ and let $E$ be the associated matrix as above.  Then 
\[ \deg_{H}(X_{k}')=(H')_{k}=(E^{T}H)_{k}=(E^{T}GM)_{k}=(E^{T}G)_{k}M=(G')_{k}M=\deg_{G}(X_{k}')M \]
as required.
\end{proof}

Therefore, to describe the degree of a cluster variable of a graded cluster algebra $\curly{A}(\underline{x},B,H)$, it suffices to know its degree with respect to some standard grading $G$ and the matrix $M=M(G,H)$ transforming $G$ to $H$.  In particular, to understand the distribution of the degrees of cluster variables, it suffices to know this for standard gradings.

Since the lemma applies in the particular case when $G$ and $H$ are both standard, we see that from one choice of basis for the kernel of $B$, we obtain complete information.  For if we chose a second basis, the change of basis matrix tells us how to transform the degrees.  Hence up to a change of basis, there is essentially only one standard grading for each seed.  

\begin{corollary} Let $(\underline{x},B)$ and $(\underline{y},C)$ be mutation equivalent seeds in a cluster algebra $\curly{A}$ and denote by $\mu_{\bullet}$ a composition of (seed) mutations such that $\mu_{\bullet}(\underline{x},B)=(\underline{y},C)$.  Let $G$ be a standard grading for $(\underline{x},B)$ and $H$ any grading for $(\underline{y},C)$.  

Then there exists an integer matrix $M$ such that for every cluster variable $Y$ of $\curly{A}$ we have
\[ \deg_{H}(Y)=\deg_{G}(Y)M. \]
\end{corollary}

\begin{proof}  Let $E_{\bullet}$ be the product of the matrices $E$ associated to $\mu_{\bullet}$.  Then $E_{\bullet}BE_{\bullet}^{T}=C$ and $(\underline{y},C,E_{\bullet}^{T}G)$ is a grading for the seed $(\underline{y},C)$.  Now $(\underline{y},C,E_{\bullet}^{T}G)$ is standard by Lemma~\ref{l:mut-of-std-is-std} and so we may apply Lemma~\ref{l:chg-of-basis-grading} to $(\underline{y},C,E_{\bullet}^{T}G)$ and $(\underline{y},C,H)$.  

That is, there exists $M$ such that $\deg_{H}(Y)=\deg_{E_{\bullet}^{T}G}(Y)M$.  But as we noted earlier $\deg_{E_{\bullet}G}(Y)$ is equal to $\deg_{G}(Y)$: the degree of $Y$ in the graded seed $(\underline{y},C,E_{\bullet}^{T}G)$ is by definition the degree of $Y$ with respect to the grading propagated from the initial graded seed $(\underline{x},B,G)$.  Hence we have that $\deg_{H}(Y)=\deg_{G}(Y)M$ as required. 
\end{proof}

Notice that $M$ is mutation invariant: once we know the respective grading matrices $H'$ and $G'$ for the \emph{same} seed (after some mutations from $(\underline{x},B,G)$ and $(\underline{y},C,H)$), $M$ is easily calculated from $H'=G'M$.  This same $M$ then compares the respective grading matrices at any seed, or indeed we can compare gradings at different seeds via the matrix $E_{\bullet}$.

In the next section, our goal is to classify gradings in finite types (with no coefficients) in the following sense: for an initial seed $(\underline{x},B)$ of finite type, we find a standard grading and establish the number of cluster variables in each degree for this $G$.  The main consequence of the above results is that the resulting distribution is essentially independent of the choices of seed and standard grading.

\vfill
\pagebreak

\section{Gradings in finite type with no coefficients}\label{s:gradings-in-finite-type}

In finite types, we have not only the Laurent phenomenon but a rather stronger statement: for certain choices of initial seed, every cluster variable is a Laurent expression in the variables from that initial seed having a special form.  This result, due to Fomin and Zelevinsky, is holds in the case that the cluster algebra has coefficients but in this section, we will concentrate only on cluster algebras of finite type without coefficients (for reasons that will become apparent later).  As such, we state a slightly simplified version of that theorem, for the no coefficients case.

Recall that to a square integer matrix $B$, we may associate a Cartan companion $A=A(B)$, by setting $a_{ii}=2$ and $a_{ij}=-\card{b_{ij}}$ if $i\neq j$.  Then in the no coefficients case, given a square skew-symmetrizable initial exchange matrix $B$, we have a Cartan companion associated to $B$ and so can associate to $B$ the root system $\Phi$ of $A(B)$.  We lose no generality by assuming $B$ yields an irreducible root system, or equivalently that the Dynkin diagram associated to $A(B)$ is connected\footnote{This reduction is universal in the literature, if rarely explicit.  It is justified by the observation that a finite type cluster algebra associated to a reducible root system is the tensor product of the cluster algebras associated to each irreducible component and when considering gradings we obtain complete information from studying the constituent cluster algebras.}

\begin{theorem}[{\cite[Theorem~1.9]{FZ-CA2}}]\label{t:FZ-Laurent-finite-type} Let $(\underline{x},B)$ be an initial seed for a cluster algebra $\curly{A}=\curly{A}(\underline{x},B)$ of finite type, such that $b_{ij}b_{ik}\geq 0$ for all $i,j,k$.  Then there is a unique bijection $\alpha \mapsto X[\alpha]$ between the almost positive roots $\Phi_{\geq 1}$ in $\Phi$ and the cluster variables in $\curly{A}$ such that, for each $\alpha \in \Phi_{\geq 1}$, the cluster variable $X[\alpha]$ is expressed in terms of the initial cluster $\underline{x}$ as 
\[ X[\alpha] = \frac{P_{\alpha}(\underline{x})}{\underline{x}^{\alpha}}, \]
where $P_{\alpha}$ is a polynomial over $\integ$ with non-zero constant term.  Under this bijection, $X[-\alpha_{i}]=X_{i}$.
\end{theorem}

As before, for $\alpha=\sum_{i=1}^{r} a_{i}\alpha_{i}$ (the expansion of $\alpha$ in the basis of simple roots $\{ \alpha_{i} \mid 1\leq i \leq r \}$) we have $\underline{x}^{\alpha} \defeq \prod_{i=1}^{r} X_{i}^{a_{i}}$.

We note that the condition on $B$, namely that $b_{ij}b_{ik}\geq 0$, corresponds in the simply laced case to choosing a quiver with a source-sink orientation on the underlying Dynkin diagram associated to $A(B)$.  Such an orientation exists because $A(B)$ is of finite type, but this condition is satisfied by a larger class of matrices than just these.  A seed with $B$ satisfying this condition is called \emph{bipartite}.

It is well-known that all orientations of a Dynkin diagram are mutation equivalent but not all orientations yield the particular form of Laurent polynomial in the above theorem: taking the linear orientation of the Dynkin diagram of type $A_{3}$, the exchange relation at the middle node is of the form $X_{2}X_{2}'=X_{1}+X_{3}$ and we do not have a non-zero constant term.

But to classify gradings on a cluster algebra, we may choose any convenient initial seed and so we choose a bipartite seed.  The following is then immediate from the above theorem.

\begin{corollary}\label{c:deg-as-fn-of-pos-roots} Let $\curly{A}(\underline{x},B,G)$ be a graded cluster algebra of finite type such that the seed $(\underline{x},B)$ is bipartite.  Then for each $\alpha \in \Phi_{\geq 0}$, expressed in the basis of simple roots as $\alpha=\sum_{i=1}^{r} a_{i}\alpha_{i}$, we have 
\[ \deg_{G}(X[\alpha])=-\alpha G=(-\sum_{i=1}^{r} a_{i}G_{i1},\dotsc ,-\sum_{i=1}^{r} a_{i}G_{id}). \]
\end{corollary}

\begin{proof} For $\alpha$ a negative simple root, the statement is immediate: the degree of $X_{i}=X[-\alpha_{i}]$ is by definition $G_{i}$ (the $i$th row of $G$).  

For the positive roots, we have from the theorem of Fomin--Zelevinsky above that $X[\alpha]=P_{\alpha}(\underline{x})/\underline{x}^{\alpha}$.  Then $X[\alpha]$ and $\underline{x}^{\alpha}$ are homogeneous with respect to $G$, the latter obviously so and the former by the fundamental property of graded cluster algebras (Proposition~\ref{p:gradedCA}).  So it follows that $P_{\alpha}(\underline{x})$ is homogeneous also and must in fact have degree $\underline{0}$ since it has non-zero constant term.

We deduce that $\deg_{G}(X[\alpha])=\deg_{G}(\underline{x}^{\alpha})=-\alpha G$, the negative sign arising because $\alpha$ is expressed in terms of the simple \emph{positive} roots, whereas $G_{i}$ is the degree of the corresponding negative simple root.  Note that in particular $X[\alpha_{i}]=-G_{i}=-X[-\alpha_{i}]=-\deg_{G}(X_{i})$.
\end{proof}

That is, in order to know the degree of a cluster variable in finite type with no coefficients, we need only know the almost positive root $\alpha$ to which it is associated and the initial grading $G$.  Then the degree in question is simply a very natural linear function in these, $-\alpha G$.  In the subsequent section, we will see a generalisation of this result, using cluster categories.

Furthermore, the results of the previous section show us that we may infer everything we wish to know from a standard grading and that then we may in fact use any initial seed we like, at the cost of altering the formula in the above Corollary by post-multiplication by some integer matrix.  Note though that in any case we certainly still obtain a linear expansion of the degrees in terms of the components of the positive root $\alpha$.

To complete our classification programme then, we calculate a basis for the kernel of an initial exchange matrix $B$ for some bipartite seed, form a standard multi-grading $G$ and calculate $-\alpha G$ for each positive root $\alpha$, using the well-known description of the sets of positive roots as sums of simple ones.

\subsection{Type $A$}\label{ss:gradings-in-type-A}

We take as initial cluster the set $\underline{x}=\{ X_{1},\dotsc ,X_{n} \} \subseteq \mathbb{Q}(X_{1},\dotsc ,X_{n})$ and initial exchange quiver $Q$ as follows:
\begin{center} 
\scalebox{1}{\begin{tikzpicture}

\node (1) at (0,2) {1};
\node (2) [right=of 1] {2};
\node (3) [right=of 2] {3};
\node (4) [right=of 3] {4};
\node (dots) [right=of 4] {$\cdots$};
\node (n) [right=of dots] {$n$};

\draw[->] (1) to (2);
\draw[<-] (2) to (3);
\draw[->] (3) to (4);
\draw[<-] (4) to (dots);
\draw[<-] (dots) to (n);

\end{tikzpicture}}
\end{center}
More precisely, we orient the Dynkin diagram of type $A_{n}$ with every odd-numbered vertex being a source and every even-numbered vertex a sink.

The exchange matrix $B(Q)$ associated to $Q$ is easily seen to have rank $n$ if $n$ is even and rank $n-1$ if $n$ is odd.  Hence if $n$ is even, the only grading is the zero grading.

Assume now that $n$ is odd.  Then the kernel of $B(Q)$ is spanned by the $n \times 1$ vector $G\in \integ^{n}$ with
\[ G_{i} = \begin{cases} 1 & \text{if}\ i\equiv 1 \bmod 4 \\ -1 & \text{if}\ i \equiv 3 \bmod 4 \\ 0 & \text{otherwise} \end{cases} \]

We deduce from our previous results that if $\alpha=\sum_{i=1}^{n} a_{i}\alpha_{i}$ is a positive root expressed in the basis of simple roots, then \[ \deg_{G}(X[\alpha])=-\alpha G=\sum_{j=1}^{\frac{n+1}{2}} (-1)^{j}a_{2j-1} \]
(Since $G$ induced a $\integ$-grading, this is of course an integer rather than a vector.)

Any positive root $\alpha$ in type $A_{n}$ is a sum of consecutive simple roots with multiplicity $1$, i.e. $\alpha=\sum_{j=1}^{l} \alpha_{i+j-1}$.  So we see that $\deg_{G}(X[\alpha])\in \{-1,0,1\}$ for all $\alpha \in \Phi_{\geq 0}$.  

Indeed it is straightforward to calculate the following distribution of degrees in type $A_{n}$ for $n$ odd:
\begin{itemize}
\item the number of cluster variables of degree $1$ is equal to $\frac{(n+1)(n+3)}{8}$,
\item the number of cluster variables of degree $0$ is equal to $\frac{(n-1)(n+3)}{4}$ and
\item the number of cluster variables of degree $-1$ is equal to $\frac{(n+1)(n+3)}{8}$.
\end{itemize}
These counts accord with the total number of cluster variables being the number of almost positive roots in a root system of type $A_{n}$, which is $\frac{n^{2}+3n}{2}$.

Let us say that a $\integ^{d}$-grading is \emph{balanced} if for all degrees $\underline{d} \in \integ^{d}$ there is a bijection between the variables of degree $\underline{d}$ and those of degree $-\underline{d}$.  (Note that this definition is valid in infinite types also.)

Then we observe that every grading for a type $A_{n}$ cluster algebra with no coefficients is balanced.  For even $n$, the zero grading is certainly balanced and for odd $n$, any initial grading vector is an integer multiple of $G$ and hence gives rise to a balanced grading.  Subsequently, we will explain why these gradings are balanced in terms of a property of the associated cluster category.

\subsection{Type $B$}\label{ss:gradings-in-type-B}

In type $B$, we have the following Dynkin diagram and associated Cartan matrix:
\begin{center}
\begin{tabular}{cc}
\scalebox{1}{\begin{tikzpicture}

\node (1) at (0,2) {$\bullet$};
\node (2) [right=of 1] {$\bullet$};
\node (3) [right=of 2] {$\bullet$};
\node (4) [right=of 3] {$\bullet$};
\node (n-1) [right=of 4] {$\bullet$};
\node (n) [right=of n-1] {$\bullet$};

\draw[-] (1) to (2);
\draw[-] (2) to (3);
\draw[-] (3) to (4);
\draw[dotted] (4) to (n-1);
\draw[double,double equal sign distance,-implies] (n-1) to (n);

\end{tikzpicture}} &
$\begin{pmatrix} 
 2 & -1 &  0 &  0 &        &    &    &    \\
-1 &  2 & -1 &  0 &        &    &    &    \\
 0 & -1 &  2 & -1 &        &    &    &    \\
   &    &    &    & \ddots &    &    &    \\
   &    &    &    &        &  2 & -1 &  0 \\
   &    &    &    &        & -1 &  2 & -1 \\
   &    &    &    &        &  0 & -2 &  2 
\end{pmatrix}$
\end{tabular}
\end{center}

We choose an exchange matrix $B$ whose Cartan companion is the above and is bipartite, specifically:
\[ \begin{pmatrix} 
 0 &  1 &  0 &  0 &        &    &    &    \\
-1 &  0 & -1 &  0 &        &    &    &    \\
 0 &  1 &  0 &  1 &        &    &    &    \\
   &    &    &    & \ddots &    &    &    \\
   &    &    &    &        &  0 &  1 &  0 \\
   &    &    &    &        & -1 &  0 & -1 \\
   &    &    &    &        &  0 &  2 &  0 
\end{pmatrix} \]
in the case that $n$ is odd, and the same but with the signs of the final rows reversed when $n$ is even.

Then this exchange matrix has rank $n$ when $n$ is even (and again we only have the zero grading) and rank $n-1$ when $n$ is odd.  

Assume now that $n$ is odd.  Then the kernel of $B^{T}$ is spanned by the $n \times 1$ vector $G\in \integ^{n}$ with
\[ G_{i} = \begin{cases} 2 & \text{if}\ i\equiv 1 \bmod 4,\ i<n \\ -2 & \text{if}\ i \equiv 3 \bmod 4,\ i<n \\ 1 & \text{if}\ i=n \\ 0 & \text{otherwise} \end{cases} \]

As above, we may deduce a (linear) formula for the degree of an arbitrary cluster variable and from this easily infer that the degrees in this case belong to the set $\{ -2, -1, 0 , 1 ,2 \}$.  The distribution of the degrees for this grading in type $B_{n}$ ($n$ odd) is (\cite{TBP})
\begin{itemize}
\item the number of cluster variables of degree $2$ is equal to $\frac{(n+1)(n-1)}{4}$,
\item the number of cluster variables of degree $1$ is equal to $\frac{n+1}{2}$,
\item the number of cluster variables of degree $0$ is equal to $\frac{(n+1)(n-1)}{4}$,
\item the number of cluster variables of degree $-1$ is equal to $\frac{n+1}{2}$ and 
\item the number of cluster variables of degree $-2$ is equal to $\frac{(n+1)(n-1)}{4}$.
\end{itemize}
These counts accord with the total number of cluster variables being the number of almost positive roots in a root system of type $B_{n}$, which is $n^2+n$.  We observe that this grading (and hence all gradings in type $B$) is again balanced.

\subsection{Type $C$}\label{ss:gradings-in-type-C}

In type $C$, we have the following Dynkin diagram and associated Cartan matrix:
\begin{center}
\begin{tabular}{cc}
\scalebox{1}{\begin{tikzpicture}

\node (1) at (0,2) {$\bullet$};
\node (2) [right=of 1] {$\bullet$};
\node (3) [right=of 2] {$\bullet$};
\node (4) [right=of 3] {$\bullet$};
\node (n-1) [right=of 4] {$\bullet$};
\node (n) [right=of n-1] {$\bullet$};

\draw[-] (1) to (2);
\draw[-] (2) to (3);
\draw[-] (3) to (4);
\draw[dotted] (4) to (n-1);
\draw[double,double equal sign distance,-implies] (n) to (n-1);

\end{tikzpicture}} &
$\begin{pmatrix} 
 2 & -1 &  0 &  0 &        &    &    &    \\
-1 &  2 & -1 &  0 &        &    &    &    \\
 0 & -1 &  2 & -1 &        &    &    &    \\
   &    &    &    & \ddots &    &    &    \\
   &    &    &    &        &  2 & -1 &  0 \\
   &    &    &    &        & -1 &  2 & -2 \\
   &    &    &    &        &  0 & -1 &  2 
\end{pmatrix}$
\end{tabular}
\end{center}

We choose an exchange matrix $B$ whose Cartan companion is the above and is bipartite, specifically:
\[ \begin{pmatrix} 
 0 &  1 &  0 &  0 &        &    &    &    \\
-1 &  0 & -1 &  0 &        &    &    &    \\
 0 &  1 &  0 &  1 &        &    &    &    \\
   &    &    &    & \ddots &    &    &    \\
   &    &    &    &        &  0 &  1 &  0 \\
   &    &    &    &        & -1 &  0 & -2 \\
   &    &    &    &        &  0 &  1 &  0 
\end{pmatrix} \]
in the case that $n$ is odd, and the same but with the signs of the final rows reversed when $n$ is even.

Again this exchange matrix has rank $n$ when $n$ is even (so we only have the zero grading) and rank $n-1$ when $n$ is odd.  

Assume now that $n$ is odd.  Then the kernel of $B^{T}$ is spanned by the $n \times 1$ vector $G\in \integ^{n}$ with
\[ G_{i} = \begin{cases} 1 & \text{if}\ i\equiv 1 \bmod 4,\ i<n \\ -1 & \text{if}\ i \equiv 3 \bmod 4,\ i<n \\ 1 & \text{if}\ i=n \\ 0 & \text{otherwise} \end{cases} \]

As above, we easily infer a formula for the degree of an arbitrary cluster variable and that the degrees in this case belong to the set $\{ -1, 0 , 1 \}$.  The distribution of the degrees for this grading in type $C_{n}$ ($n$ odd) is (\cite{TBP})
\begin{itemize}
\item the number of cluster variables of degree $1$ is equal to $\left(\frac{n+1}{2}\right)^{2}$,
\item the number of cluster variables of degree $0$ is equal to $\frac{(n+1)(n-1)}{2}$ and
\item the number of cluster variables of degree $-1$ is equal to $\left(\frac{n+1}{2}\right)^{2}$. 
\end{itemize}
We observe that this grading is again balanced, and thus so are all gradings in type $C$.

\subsection{Type $D$}\label{ss:gradings-in-type-D}

We now turn our attention to type $D$, taking as our initial quiver
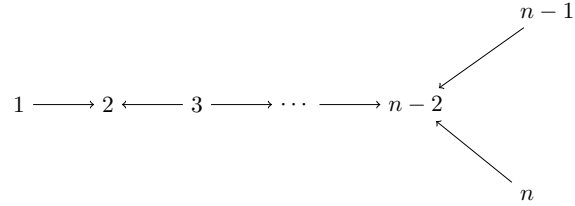
\begin{center}
\scalebox{0.8}{\begin{tikzpicture}

\node (1) at (0,2) {1};
\node (2) [right=of 1] {2};
\node (3) [right=of 2] {3};
\node (dots) [right=of 3] {$\cdots$};
\node (n-2) [right=of dots] {$n-2$};
\node (n-1) [above right=of n-2] {$n-1$};
\node (n) [below right=of n-2] {$n$};

\draw[->] (1) to (2);
\draw[<-] (2) to (3);
\draw[->] (3) to (dots);
\draw[->] (dots) to (n-2);
\draw[->] (n-1) to (n-2);
\draw[->] (n) to (n-2);

\end{tikzpicture}}
\end{center}
when $n$ is even; the vertex $n-2$ is instead a source for $n$ odd.  The corresponding exchange matrix $B$ has rank $n-1$ if $n$ is odd and rank $n-2$ if $n$ is even, so we have non-zero gradings in all cases.  Indeed, the even $n$ case gives our first example of a $\integ^{2}$-grading.

For odd $n$, the kernel of $B^{T}$ is spanned by the vector $G$ with 
\[ G_{i} = \begin{cases} 0 & \text{if}\ i\leq n-2 \\ 1 & \text{if}\ i=n-1 \\ -1 & \text{if}\ i=n \end{cases} \]
Clearly the degree of any cluster variable $X[\alpha]$ with $\alpha=\sum_{i=1}^{n} a_{i}\alpha_{i}$ is $-\alpha G=a_{n}-a_{n-1}$.  From this and the well-known description of the positive roots in type $D$, we deduce that the degrees in this case belong to the set $\{-1,0,1\}$ and
\begin{itemize}
\item the number of cluster variables of degree $1$ is equal to $n$,
\item the number of cluster variables of degree $0$ is equal to $n(n-2)$ and
\item the number of cluster variables of degree $-1$ is equal to $n$. 
\end{itemize}

For even $n$, we know that it suffices to study the standard grading $G$.  This case breaks down further according to whether $n$ is congruent to $0$ or $2$ modulo 4.  If $n\equiv 0 \bmod 4$, $G$ is given by
\[ G_{i} = \begin{cases} (1,0) & \text{if}\ i\equiv 1 \bmod 4,\ i<n-1 \\ (0,0) & \text{if}\ i\ \text{even}, i<n-1 \\ (-1,0) & \text{if}\ i\equiv 3 \bmod 4,\ i<n-1 \\ (-1,1) & \text{if}\ i=n-1 \\ (0,-1) & \text{if}\ i=n \end{cases} \]
If $n\equiv 2 \bmod 4$, we have
\[ G_{i} = \begin{cases} (1,0) & \text{if}\ i\equiv 1 \bmod 4,\ i<n-1 \\ (0,0) & \text{if}\ i\ \text{even}, i<n-1 \\ (-1,0) & \text{if}\ i\equiv 3 \bmod 4,\ i<n-1 \\ (1,-1) & \text{if}\ i=n-1 \\ (0,1) & \text{if}\ i=n \end{cases} \]

Analysing the resulting formula $\deg_{G}(X[\alpha])=-\alpha G$, we see that the (multi-)degrees in this case belong to the set $(\{ -1, 0 , 1\} \cross \{ -1 , 0 , 1\}) \setminus \{ (1,1), (-1,-1)\}$.  The corresponding distribution of degrees is given in Table~\ref{table:degrees-for-D-2m}.  We note that the distribution does not in fact depend on the congruence of $n$ modulo $4$ and also that this bi-grading is also balanced: the number of variables of degree $\underline{d}$ is equal to the number of degree $-\underline{d}$.

\begin{table}
\begin{center}
\begin{tabular}{|r|c|c|c||c|} \hline
\rule[-0.5em]{0em}{1.5em}  & $-1$ & $0$ & $1$ & Total \\ \hline
\rule[-1em]{0em}{2.5em} $-1$ &  $0$  & $\frac{n^{2}-2n}{4}$ & $\frac{n}{2}$ & $\frac{n^{2}}{4}$ \\ \hline
\rule[-1em]{0em}{2.5em}  $0$ & $\frac{n}{2}$  & $\frac{n^{2}-2n}{2}$ & $\frac{n}{2}$ & $\frac{n^{2}}{2}$  \\ \hline
\rule[-1em]{0em}{2.5em}  $1$ & $\frac{n}{2}$  & $\frac{n^{2}-2n}{4}$ & $0$ & $\frac{n^{2}}{4}$ \\ \hline \hline
\rule[-0.5em]{0em}{1.5em} Total & $n$  & $n^{2}-2n$ & $n$ & $n^{2}$ \\ \hline
\end{tabular}
\end{center}
\caption{Distribution of degrees for type $D_{n}$, $n$ even: the $(a,b)$-entry of the table gives the number of cluster variables of degree $(a,b)$.\label{table:degrees-for-D-2m}}
\end{table}

\subsection{Types $E$, $F$ and $G$}\label{ss:gradings-in-type-E-F-G}

One easily checks that the exchange matrices of type $G_{2}$, $F_{4}$, $E_{6}$ and $E_{8}$ have maximal rank, so that the coefficient-free cluster algebras of these types admit no non-zero gradings.

However, exchange matrices of type $E_{7}$ have rank $6$, so we do have a grading in this case.  For the quiver
\begin{center}
\scalebox{0.8}{\begin{tikzpicture}

\node (1) at (0,2) {1};
\node (3) [right=of 1] {3};
\node (4) [right=of 3] {4};
\node (2) [above=of 4] {2};
\node (5) [right=of 4] {5};
\node (6) [right=of 5] {6};
\node (7) [right=of 6] {7};

\draw[->] (1) to (3);
\draw[->] (4) to (3);
\draw[->] (4) to (5);
\draw[->] (6) to (5);
\draw[->] (6) to (7);
\draw[->] (4) to (2);
\end{tikzpicture}}
\end{center}
we have the grading
\begin{center}
\scalebox{0.8}{\begin{tikzpicture}

\node (1) at (0,2) {$0$};
\node (3) [right=of 1] {$0$};
\node (4) [right=of 3] {$0$};
\node (2) [above=of 4] {$1$};
\node (5) [right=of 4] {$-1$};
\node (6) [right=of 5] {$0$};
\node (7) [right=of 6] {$1$};

\draw[->] (1) to (3);
\draw[->] (4) to (3);
\draw[->] (4) to (5);
\draw[->] (6) to (5);
\draw[->] (6) to (7);
\draw[->] (4) to (2);
\end{tikzpicture}}
\end{center}

By computer-aided calculation of the cluster variables in this case, we find that this grading has 
\begin{itemize}
\item 15 cluster variables in degree $1$,
\item 40 in degree $0$ and
\item 15 in degree $-1$.
\end{itemize}

This concludes our analysis of the finite type cases with no coefficients.  We now turn our attention to cluster categories and gradings on these, in order to mirror and extend the above results in the categorical setting.  In particular we will explain the repeated occurrences of balanced gradings.

\section{Graded cluster categories}\label{s:graded-cluster-categories}

We wish to lift the notion of a multi-graded cluster algebra to the setting of (generalised) cluster categories.  Doing so, we obtain categorical versions of the results of the previous sections and indeed gain further insight and generalisations.  A significant advantage to working with a cluster category is that it gives a global picture of the cluster combinatorics, which is particularly helpful when examining gradings.

We make use of recent results of Dominguez and Gei\ss\ (\cite{DominguezGeiss}), generalising earlier work of Caldero--Chapoton (\cite{CalderoChapoton}), Palu (\cite{Palu}) and others.  The following constructions are in the skew-symmetric and no coefficients settings but we are no longer assuming finite type.  We recap the necessary setup from \cite{DominguezGeiss} and adopt the conventions there.

\begin{definition} Let $\mathbb{K}$ be an algebraically closed field.  Let $\curly{C}$ be a triangulated 2-Calabi--Yau $\mathbb{K}$-category with suspension functor $\Sigma$ and let $T\in \curly{C}$ be a basic cluster-tilting object.  We will call the pair $(\curly{C},T)$ a generalised cluster category.
\end{definition}

Following the nomenclature of Assem--Dupont--Schiffler (\cite{ADS}), we might rather call $(\curly{C},T)$ a generalised \emph{rooted} cluster category, as the analogue of the initial seed is required to be part of the data, but for brevity we shall not.

Write $T=T_{1}\dsum \dotsm \dsum T_{r}$.  Setting $\Lambda=\op{\End{\curly{C}}{T}}$, the functor $E=\curly{C}(T,-)\colon \curly{C} \to \Lambda\text{-mod}$ induces an equivalence $\curly{C}/\text{add}(T)\to \Lambda\text{-mod}$.  We may also define an exchange matrix associated to $T$ by
\[ (B_{T})_{ij}=\dim \text{Ext}_{\Lambda}^{1}(S_{i},S_{j})-\dim \text{Ext}_{\Lambda}^{1}(S_{j},S_{i}) \]
Here the $S_{i}=E\Sigma^{-1}T_{i}/\text{rad}\,E\Sigma^{-1}T_{i}$, $i=1,\dotsc ,r$ are the simple $\Lambda$-modules.

Then for each $X\in \curly{C}$ there exists a distinguished triangle
\[ \bigdsum_{i=1}^{r} T_{i}^{m(i,X)} \to \bigdsum_{i=1}^{r} T_{i}^{p(i,X)} \to X \to \Sigma \left( \bigdsum_{i=1}^{r} T_{i}^{m(i,X)} \right) \]
Define the index of $X$ with respect to $T$, $\ind{T}{X}$, to be the integer vector with $\ind{T}{X}_{i}=p(i,X)-m(i,X)$.  By \cite[\S 2.1]{Palu}, $\ind{T}{X}$ is well-defined and taking $\mathbb{K}=\complex$ we have a cluster character
\begin{align*} C_{?}^{T} \colon \text{Obj}(\curly{C}) &\to \mathbb{Q}[X_{1}^{\pm 1},\dotsc ,X_{r}^{\pm 1}] \\
 X & \mapsto \underline{x}^{\ind{T}{X}}\sum_{\underline{e}} \chi(\mathrm{Gr}_{\underline{e}}(EX))\underline{x}^{B_{T}\cdot\underline{e}}
\end{align*}
Here $\mathrm{Gr}_{\underline{e}}(EX)$ is the quiver Grassmannian of $\Lambda$-submodules of $EX$ of dimension vector $\underline{e}$ and $\chi$ is the topological Euler characteristic.  We also use the same monomial notation ($\underline{x}^{\underline{a}}$) as previously.  We note that this formula generalises that of Fomin--Zelevinsky recalled in Theorem~\ref{t:FZ-Laurent-finite-type}.

We also recall that for any cluster-tilting object $U$ of $\curly{C}$ and for any $U_{k}$ an indecomposable summand of $U$, there exists a unique indecomposable object $U_{k}^{*}\not\iso U_{k}$ such that $U^{*}=(U/U_{k})\dsum U_{k}^{*}$ is again cluster-tilting and there exist non-split triangles
\[ U_{k}^{*} \to M \to U_{k} \to \Sigma U_{k}^{*} \qquad \text{and} \qquad U_{k} \to M' \to U_{k}^{*} \to \Sigma U_{k} \]
with $M,M' \in \text{add}(U/U_{k})$.  In the generality of our setting, this is due to Iyama and Yoshino (\cite{IyamaYoshino}).

The obvious definition of a graded generalised cluster category is then the following.

\begin{definition} Let $(\curly{C},T)$ be a generalised cluster category and let $G$ be an $r\cross d$ integer matrix such that $B_{T}G=0$.  We call the tuple $(\curly{C},T,G)$ a graded generalised cluster category.
\end{definition}

\begin{definition} Let $(\curly{C},T,G)$ be a graded generalised cluster category.  For any $X\in \curly{C}$, we define $\deg_{G}(X)=\ind{T}{X}G$.
\end{definition}

These definitions are justified by the following proposition.

\begin{proposition}\label{p:prop-of-gen-cc} Let $(\curly{C},T,G)$ be a graded generalised cluster category.
\begin{enumerate}[label=(\roman*)]
\item For all $X \in \curly{C}$, $\deg_{G}(X)$ is equal to the degree of $C_{X}^{T}\in \mathbb{Q}[X_{1}^{\pm 1},\dotsc ,X_{r}^{\pm 1}]$ where the latter is $\integ^{d}$-graded by $\deg_{G}(X_{i})=G_{i}$ (the $i$th row of $G$).
\item For any distinguished triangle in $\curly{C}$, $X\to Y \to Z \to \Sigma X$, we have
\[ \deg_{G}(Y)=\deg_{G}(X)+\deg_{G}(Z) \]
\item The degree $\deg_{G}$ is compatible with mutation in the sense that for every cluster-tilting object $U$ of $\curly{C}$ with indecomposable summand $U_{k}$ we have
\[ \deg_{G}(U_{k}^{*})=\deg_{G}(M)-\deg_{G}(U_{k})=\deg_{G}(M')-\deg_{G}(U_{k}) \]
where $U_{k}^{*}$, $M$ and $M'$ are as in the above description of mutation in $\curly{C}$.
\end{enumerate}
\end{proposition}

\begin{proof} {\ }
\begin{enumerate}[label=(\roman*)]
\item\label{pf:prop-of-gen-cc-1} For $\alpha=(a_{1},\dotsc ,a_{r})\in \integ^{r}$ the degree of $\underline{x}^{\alpha}$ is $\sum_{i=1}^{r} a_{i}G_{i}=\alpha G$.  Hence for each $\underline{e}$, $\deg_{G}(\underline{x}^{B_{T}\cdot \underline{e}})=(B_{T}\cdot \underline{e})G=0$ since $B_{T}G=0$.  It follows that \[ \deg_{G}(C_{X}^{T})=\deg_{G}(\underline{x}^{\ind{T}{X}})=\ind{T}{X}G=\deg_{G}(X). \]
\item\label{pf:prop-of-gen-cc-2} By \cite[Proposition~2.5(b)]{DominguezGeiss}, for any distinguished triangle in $\curly{C}$, $X\stackrel{\alpha}{\to} Y \stackrel{\beta}{\to} Z \stackrel{\gamma}{\to} \Sigma X$,
\[ \ind{T}{Y}=\ind{T}{X}+\ind{T}{Z}+B_{T}\cdot \dimvec_{\Lambda}(\ker E\alpha) \]
Multiplying by $G$ on the right we immediately deduce that 
\[ \deg_{G}(Y)=\deg_{G}(X)+\deg_{G}(Z) \]
since $(B_{T}\cdot \dimvec_{\Lambda}(\ker E\alpha ))G=0$, similarly to \ref{pf:prop-of-gen-cc-1}.
\item This is immediate from the above description of mutation of cluster-tilting objects in $\curly{C}$ and \ref{pf:prop-of-gen-cc-2}.  We remark that $\deg_{G}(M)=\deg_{G}(M')$, which is the categorical version of the claim that all exchange relations in a graded cluster algebra are homogeneous. \qedhere
\end{enumerate}
\end{proof}

Observe that the degree of an object $X$ in a graded generalised cluster category is a linear function of $\ind{T}{X}$, namely $\ind{T}{X}G$.  This is a generalisation of Corollary~\ref{c:deg-as-fn-of-pos-roots} to arbitrary types, where the collection of vectors $\{ \ind{T}{X} \mid X\in \curly{C},\ X\ \text{indecomposable}\,\}$ replaces the set of almost positive roots.  Indeed, we may deduce Corollary~\ref{c:deg-as-fn-of-pos-roots} from well-known properties of cluster categories of finite type.

Given a triangulated category $\curly{A}$, we may form its Grothendieck group $K_{0}(\curly{A})$ as the group generated by isoclasses of objects, $[A]$ for $A\in \curly{A}$, modulo relations $[X]-[Y]+[Z]=0$ for every distinguished triangle $X\to Y\to Z\to \Sigma X$.

We have the following remarkable characterisation of gradings for generalised cluster categories:

\begin{proposition} The space of gradings for a generalised cluster category $(\curly{C},T)$ may be identified with the Grothendieck group $K_{0}(\curly{C})$.
\end{proposition}

\begin{proof} Let $\curly{T}=\text{proj}\, \Lambda$ be the category of finitely generated projective $\Lambda$-modules.  Then $K_{0}(\curly{T})$ is free abelian on the basis $\{ [T_{i}] \}$.  By work of Palu (\cite{Palu-Groth-gp}), $K_{0}(\curly{C})$ is isomorphic to the quotient of $K_{0}(\curly{T})$ by all relations $[M]=[M']$ where
\[ U_{k}^{*} \to M \to U_{k} \to \Sigma U_{k}^{*} \qquad \text{and} \qquad U_{k} \to M' \to U_{k}^{*} \to \Sigma U_{k} \]
are the non-split triangles associated to mutation of cluster-tilting objects in $(\curly{C},T)$.  But the latter is easily identified with $\text{Im}\, B_{T}$.  That is, $K_{0}(\curly{C})$ may be identified with $\text{Ker}\, B_{T}$, the space of gradings.
\end{proof}

Notice that this implies, for example, that the Grothendieck group of the cluster category of type $A_{2m}$ is trivial.  Similarly $K_{0}(\curly{C}_{A_{2m+1}})\iso \integ$.  These calculations of the Grothendieck groups in finite types had been found by Barot--Kussin--Lenzing (\cite{BKL}).

However we obtain even more from the categorical approach.  Recall that $\curly{C}$ has a suspension (or shift) functor $\Sigma$ that is an automorphism of $\curly{C}$.  This additional symmetry of $\curly{C}$ induces the following property of $\deg_{G}$.

\begin{lemma} For each $X\in \curly{C}$, $\deg_{G}(\Sigma X)=-\deg_{G}(X)$.  

That is, for each $\underline{d}\in \integ^{d}$, $\Sigma$ induces a bijection between the objects of $\curly{C}$ of degree $\underline{d}$ and those of degree $-\underline{d}$.
\end{lemma}

\begin{proof} By \cite[Proposition~2.5(a)]{DominguezGeiss} we have that
\[ -B_{T}\cdot \dimvec_{\Lambda}(EX)=\ind{T}{X}+\ind{T}{\Sigma X} \]
from which the claim follows similarly to above, by post-multiplication by $G$.
\end{proof}

Note that as a consequence any object $X$ for which $\Sigma^{2m+1}X=X$ must have degree $\underline{0}$.

Let us say that a cluster algebra $\curly{A}=\curly{A}(Q) \subseteq \mathbb{Q}[X_{1}^{\pm 1},\dotsc ,X_{r}^{\pm 1}]$ arising from a quiver $Q$ (with no coefficients) admits a cluster categorification $(\curly{C}(\curly{A}),T)$ if the cluster character \linebreak $C_{?}^{T}\colon \curly{C}(\curly{A}) \to \mathbb{Q}[X_{1}^{\pm 1},\dotsc ,X_{r}^{\pm 1}]$ is a bijection between indecomposable objects of $\curly{C}(\curly{A})$ and the cluster variables of $\curly{A}$.  In particular, by work of Palu (\cite{Palu}), if $Q$ is a finite connected\footnote{This restriction is again mild, similarly to before.} acyclic quiver (that is, $Q$ is mutation equivalent to a quiver having no oriented cycles) then its associated cluster algebra $\curly{A}(Q)$ admits a cluster categorification, with $\curly{C}(\curly{A})=\curly{C}_{Q}$ the usual cluster category introduced in \cite{BMRRT}.

\begin{corollary} Let $\curly{A}(Q)$ be a cluster algebra admitting a cluster categorification.  Then every grading for $\curly{A}(Q)$ is balanced. \qed
\end{corollary}

In particular, every grading in types $A$, $D$ and $E$ is balanced.  For type $B$, Buan, Marsh and Vatne (\cite{BMV}) have introduced a categorification, the cluster category associated to a tube, so this case is also explained.  Since by \cite{FZ-CA2}, the cluster variables in type $B_{n-1}$ and $C_{n-1}$ are \emph{both} in bijection with centrally-symmetric pairs of diagonals of a regular $2n$-gon and the aforementioned work of \cite{BMV} goes via this bijection also, we see that type $C$ is covered also.  That is, for every exchange matrix of finite type, every grading is balanced, explaining our previous observations of this fact.

\section{Tropical friezes}\label{s:tropical-friezes}

We recall the notion of a frieze pattern, as introduced by Conway and Coxeter (\cite{Coxeter}, \cite{ConwayCoxeter}).  A frieze pattern of order $n$ consists of $n-1$ infinite rows of positive integers, with the first and last rows consisting of only the integer $1$ and such that in a diamond of adjacent integers
\begin{center}
\begin{tikzpicture}[node distance=1cm,on grid]
\node (11) at (0,0) {$b$}; 
\node (21) [below left=of 11] {$a$};
\node (22) [below right=of 11] {$d$};
\node (31) [below left=of 22] {$c$};
\end{tikzpicture}
\end{center}
we have $ad-bc=1$.  The latter property is called the unimodular rule.

Following Fock--Goncharov (\cite{FockGoncharov}) and Propp (\cite{Propp}), Guo (\cite{Guo}) has studied \emph{tropical friezes} on generalised cluster categories.  A tropical frieze on $\curly{C}$ is a map $f\colon \curly{C} \to \integ$ which is constant on isomorphism classes, additive with respect to direct sums and where the unimodular rule is replaced by $a+d=\max(b+c,0)$.  More precisely, for all objects $U$ and $V$ of $\curly{C}$ with $\dim \text{Ext}_{\curly{C}}^{1}(U,V)=1$, $f(U)+f(V)=\max \{ f(M),f(M')\}$ where we have non-split triangles
\[ U \to M \to V \to \Sigma U \qquad \text{and} \qquad V \to M' \to U \to \Sigma V \]
Note that it is clear how to extend this definition to that of a ``multi-frieze'', that is when $f$ takes values in $\integ^{d}$ for $d\geq 1$.  We will continue to use the term ``frieze'' to encompass multi-friezes also.

In \cite{Guo}, Guo shows that the sum of two tropical friezes need not be again a tropical frieze and gives a necessary and sufficient condition for this to be the case.  Furthermore, she shows that if $Q$ is a Dynkin quiver with $n$ vertices and if $\curly{C}=\curly{C}_{Q}$ is its cluster category, then tropical friezes on $\curly{C}$ are in bijection with elements of $\integ^{n}$, by showing that each tropical frieze is determined by its values on the indecomposable summands of a basic cluster-tilting object and all choices are permitted.

The connection with the previous section of this work is as follows.  Let us say that a tropical frieze $f$ is \emph{exact} if for all objects $U$ and $V$ with $\dim \text{Ext}_{\curly{C}}^{1}(U,V)=1$ we have $f(U)+f(V)=f(M)=f(M')$ where we have non-split triangles involving $U$, $V$, $M$ and $M'$ as above.

\begin{lemma} Let $(\curly{C},T,G)$ be a graded generalised cluster category.  Then $\deg_{G}\colon \text{Obj}(\curly{C}) \to \integ^{d}$ is an exact tropical frieze on $\curly{C}$.  Conversely, if $f$ is an exact tropical frieze on $(\curly{C},T)$ then $(\curly{C},T,(f(T_{1}),\dotsc ,f(T_{r})))$ is a graded generalised cluster category.
\end{lemma}

\begin{proof} The first claim is immediate from Proposition~\ref{p:prop-of-gen-cc}.  The converse follows from the result of \cite{DominguezGeiss} recalled in part \ref{pf:prop-of-gen-cc-2} of the proof of that Proposition.
\end{proof}

We note that exact tropical friezes are better behaved than tropical friezes generally: in particular, the sum of any two exact tropical friezes is again a tropical frieze.  Furthermore the exact tropical friezes are classified by means of linear algebra, by finding a basis for the kernel of the associated exchange matrix $B_{T}$.  We note that Guo has classified \emph{all} tropical friezes when $Q$ is Dynkin (\cite[Theorem~5.1]{Guo}) in categorical terms; the classification for exact tropical friezes is elementary but this applies to a restricted class of friezes.

We conclude this section with two examples of exact tropical friezes that illustrate the results of this and the two previous sections in types $A_{5}$ and $D_{4}$.  In both cases we give the Auslander--Reiten quiver of the associated cluster category (Figures~\ref{fig:AR-quiver-A5} and \ref{fig:AR-quiver-D4}) and the grading---or equivalently exact tropical frieze---corresponding to the choices of bipartite quiver $Q$ and grading $G$ in Sections~\ref{ss:gradings-in-type-A} and \ref{ss:gradings-in-type-D} respectively (Figures~\ref{fig:AR-quiver-A5-with-degrees} and \ref{fig:AR-quiver-D4-bigraded}).  We see not only the tropical unimodular rule in the additivity on meshes but also the sign-changing bijection induced by $\Sigma$ (which is perhaps more familiar as the Auslander--Reiten translation $\tau$).

\begin{figure}
\begin{center} 
\scalebox{1}{\input{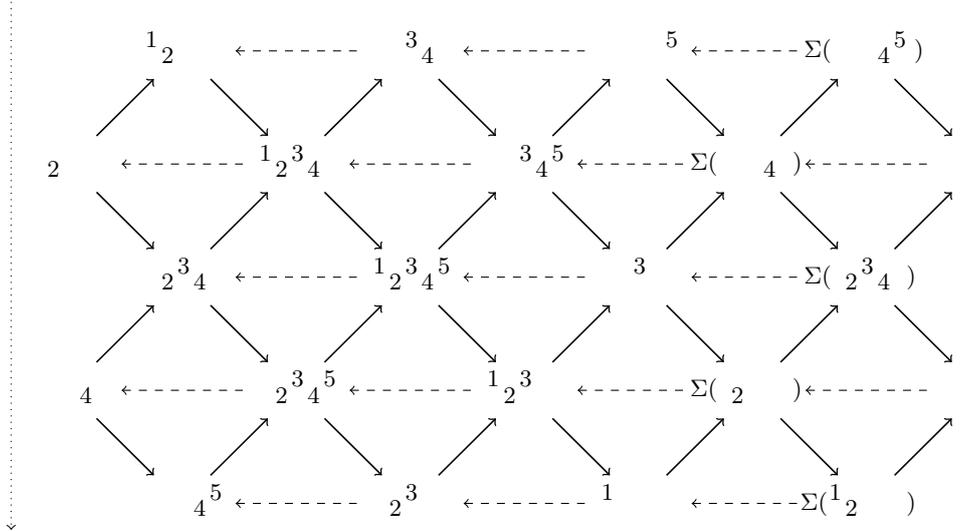}}
\end{center}
\caption{The Auslander--Reiten quiver for the cluster category of type $A_5$.  Each quiver representation is given in terms of its Loewy series, writing just ``$i$'' for the simple module $S_{i}$.  (Note that the left- and right-hand edges are to be identified at the dotted line, with a twist.)}\label{fig:AR-quiver-A5}
\vspace{1em}
\hrule
\end{figure}

\begin{figure}
\begin{center} 
\scalebox{0.9}{\input{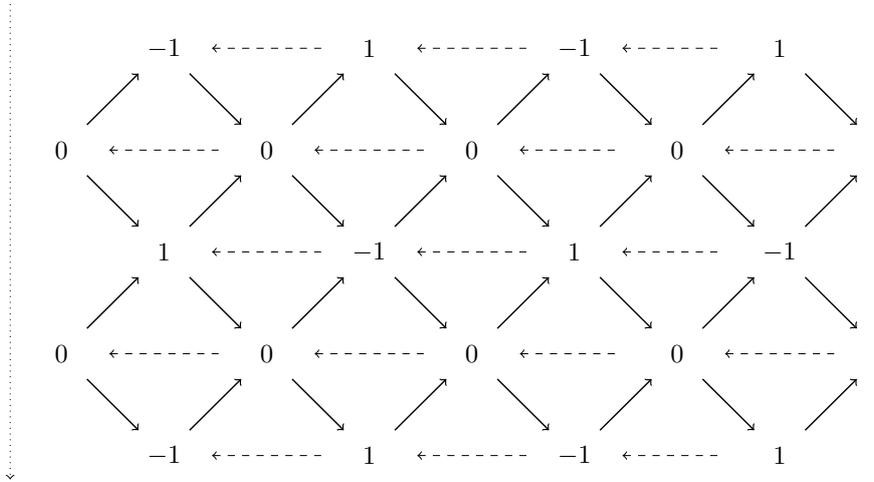}}
\end{center}
\caption{The Auslander--Reiten quiver for the cluster category of type $A_5$ with degrees replacing modules.}\label{fig:AR-quiver-A5-with-degrees}
\vspace{1em}
\hrule
\end{figure}

\begin{figure}
\begin{center} 
\scalebox{0.8}{\begin{tikzpicture}

\begin{scope}[node distance=0.3cm,on grid,align=center,text width=1em,font=\footnotesize]
\node (a1) at (0,0) {};
\node (a2) [right=of a1] {2};
\node (a3) [above right=of a2] {};
\node (a4) [below right=of a2] {};
\end{scope}

\begin{scope}[node distance=0.3cm,on grid,align=center,text width=1em,font=\footnotesize,xshift=1.5cm,yshift=1.5cm]
\node (b1) at (0,0) {};
\node (b2) [right=of b1] {2};
\node (b3) [above right=of b2] {3};
\node (b4) [below right=of b2] {};
\end{scope}

\begin{scope}[node distance=0.3cm,on grid,align=center,text width=1em,font=\footnotesize,xshift=1.5cm,yshift=0.5cm]
\node (c1) at (0,0) {1};
\node (c2) [right=of c1] {2};
\node (c3) [above right=of c2] {};
\node (c4) [below right=of c2] {};
\end{scope}

\begin{scope}[node distance=0.3cm,on grid,align=center,text width=1em,font=\footnotesize,xshift=1.5cm,yshift=-1.5cm]
\node (d1) at (0,0) {};
\node (d2) [right=of d1] {2};
\node (d3) [above right=of d2] {};
\node (d4) [below right=of d2] {4};
\end{scope}

\begin{scope}[node distance=0.3cm,on grid,align=center,text width=1em,font=\footnotesize,xshift=3cm,yshift=0cm]
\node (e1) at (0,0) {$1$};
\node (e2) [right=of e1] {$2^{2}$};
\node (e3) [above right=of e2,xshift=2pt] {$3$};
\node (e4) [below right=of e2,xshift=2pt] {$4$};
\end{scope}

\begin{scope}[node distance=0.3cm,on grid,align=center,text width=1em,font=\footnotesize,xshift=4.5cm,yshift=1.5cm]
\node (f1) at (0,0) {1};
\node (f2) [right=of f1] {2};
\node (f3) [above right=of f2] {};
\node (f4) [below right=of f2] {4};
\end{scope}

\begin{scope}[node distance=0.3cm,on grid,align=center,text width=1em,font=\footnotesize,xshift=4.5cm,yshift=0.5cm]
\node (g1) at (0,0) {};
\node (g2) [right=of g1] {2};
\node (g3) [above right=of g2] {3};
\node (g4) [below right=of g2] {4};
\end{scope}

\begin{scope}[node distance=0.3cm,on grid,align=center,text width=1em,font=\footnotesize,xshift=4.5cm,yshift=-1.5cm]
\node (h1) at (0,0) {1};
\node (h2) [right=of h1] {2};
\node (h3) [above right=of h2] {3};
\node (h4) [below right=of h2] {};
\end{scope}

\begin{scope}[node distance=0.3cm,on grid,align=center,text width=1em,font=\footnotesize,xshift=6cm,yshift=0cm]
\node (i1) at (0,0) {1};
\node (i2) [right=of i1] {2};
\node (i3) [above right=of i2] {3};
\node (i4) [below right=of i2] {4};
\end{scope}

\begin{scope}[node distance=0.3cm,on grid,align=center,text width=1em,font=\footnotesize,xshift=7.5cm,yshift=1.5cm]
\node (j1) at (0,0) {};
\node (j2) [right=of j1] {};
\node (j3) [above right=of j2] {3};
\node (j4) [below right=of j2] {};
\end{scope}

\begin{scope}[node distance=0.3cm,on grid,align=center,text width=1em,font=\footnotesize,xshift=7.5cm,yshift=0.5cm]
\node (k1) at (0,0) {1};
\node (k2) [right=of k1] {};
\node (k3) [above right=of k2] {};
\node (k4) [below right=of k2] {};
\end{scope}

\begin{scope}[node distance=0.3cm,on grid,align=center,text width=1em,font=\footnotesize,xshift=7.5cm,yshift=-1.5cm]
\node (l1) at (0,0) {};
\node (l2) [right=of l1] {};
\node (l3) [above right=of l2] {};
\node (l4) [below right=of l2] {4};
\end{scope}

\begin{scope}[node distance=0.3cm,on grid,align=center,text width=1em,font=\footnotesize,xshift=9cm]
\node (m0) at (-0.25cm,0) {$\Sigma($};
\node (m1) at (0,0) {};
\node (m2) [right=of m1] {2};
\node (m3) [above right=of m2] {};
\node (m4) [below right=of m2] {};
\node (m5) [below right=of m3,xshift=0.1cm] {$)$};
\end{scope}

\begin{scope}[node distance=0.3cm,on grid,align=center,text width=1em,font=\footnotesize,xshift=11cm,yshift=1.5cm]
\node (n0) at (-0.25cm,0) {$\Sigma($};
\node (n1) at (0,0) {};
\node (n2) [right=of n1] {2};
\node (n3) [above right=of n2] {3};
\node (n4) [below right=of n2] {};
\node (n5) [below right=of n3,xshift=0.1cm] {$)$};
\end{scope}

\begin{scope}[node distance=0.3cm,on grid,align=center,text width=1em,font=\footnotesize,xshift=11cm,yshift=0.5cm]
\node (o0) at (-0.3cm,0) {$\Sigma($};
\node (o1) at (0,0) {1};
\node (o2) [right=of o1] {2};
\node (o3) [above right=of o2] {};
\node (o4) [below right=of o2] {};
\node (o5) [below right=of o3,xshift=0.1cm] {$)$};
\end{scope}

\begin{scope}[node distance=0.3cm,on grid,align=center,text width=1em,font=\footnotesize,xshift=11cm,yshift=-1.5cm]
\node (p0) at (-0.25cm,0) {$\Sigma($};
\node (p1) at (0,0) {};
\node (p2) [right=of p1] {2};
\node (p3) [above right=of p2] {};
\node (p4) [below right=of p2] {4};
\node (p5) [below right=of p3,xshift=0.1cm] {$)$};
\end{scope}

\begin{scope}[node distance=0.3cm,on grid,align=center,text width=1em,font=\footnotesize,xshift=13cm]
\node (q1) at (0,0) {};
\end{scope}

\draw[->,semithick] (a3) to (b1);
\draw[->,semithick] (a2) to (c1);
\draw[->,semithick] (a4) to (d1);

\draw[->,semithick] (b4) to (e1);
\draw[->,semithick] (c2) to (e1);
\draw[->,semithick] (d3) to (e1);

\draw[->,semithick] (e3) to (f1);
\draw[->,semithick,shorten <=3pt] (e2) to (g1);
\draw[->,semithick] (e4) to (h1);

\draw[->,semithick] (f4) to (i1);
\draw[->,semithick] (g2) to (i1);
\draw[->,semithick] (h3) to (i1);

\draw[->,semithick] (i3) to (j1);
\draw[->,semithick] (i2) to (k1);
\draw[->,semithick] (i4) to (l1);

\draw[->,semithick,shorten >=10pt] (j4) to (m1);
\draw[->,semithick,shorten >=10pt] (k2) to (m1);
\draw[->,semithick,shorten >=10pt] (l3) to (m1);

\draw[->,semithick,shorten >=10pt] (m5) to (n1);
\draw[->,semithick,shorten >=10pt] (m5) to (o1);
\draw[->,semithick,shorten >=10pt] (m5) to (p1);

\draw[->,semithick] (n5) to (q1);
\draw[->,semithick] (o5) to (q1);
\draw[->,semithick] (p5) to (q1);

\end{tikzpicture}}
\end{center}
\caption{The Auslander--Reiten quiver for the cluster category of type $D_4$.  As usual, we have arranged the diagram so that modules in the same orbit under the shift are aligned on the same horizontal line.  (Note that the left- and right-hand edges are to be identified, in the sense that the arrows at the right edge should be regarded as pointing to the representation on the far left.)}\label{fig:AR-quiver-D4}
\vspace{1em}
\hrule
\end{figure}
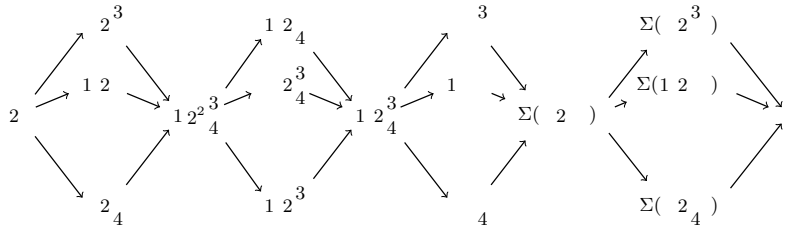

\begin{figure}
\begin{center} 
\scalebox{0.8}{\begin{tikzpicture}

\begin{scope}[node distance=0.3cm,on grid,align=center,font=\footnotesize]
\node (a1) at (0,0) {$(0,0)$};
\end{scope}

\begin{scope}[node distance=0.3cm,on grid,align=center,font=\footnotesize,xshift=1.5cm,yshift=1.5cm]
\node (b1) at (0,0) {$(1,-1)$};
\end{scope}

\begin{scope}[node distance=0.3cm,on grid,align=center,font=\footnotesize,xshift=1.5cm,yshift=0.5cm]
\node (c1) at (0,0) {$(-1,0)$};
\end{scope}

\begin{scope}[node distance=0.3cm,on grid,align=center,font=\footnotesize,xshift=1.5cm,yshift=-1.5cm]
\node (d1) at (0,0) {$(0,1)$};
\end{scope}

\begin{scope}[node distance=0.3cm,on grid,align=center,font=\footnotesize,xshift=3cm,yshift=0cm]
\node (e1) at (0,0) {$(0,0)$};
\end{scope}

\begin{scope}[node distance=0.3cm,on grid,align=center,font=\footnotesize,xshift=4.5cm,yshift=1.5cm]
\node (f1) at (0,0) {$(-1,1)$};
\end{scope}

\begin{scope}[node distance=0.3cm,on grid,align=center,font=\footnotesize,xshift=4.5cm,yshift=0.5cm]
\node (g1) at (0,0) {$(1,0)$};
\end{scope}

\begin{scope}[node distance=0.3cm,on grid,align=center,font=\footnotesize,xshift=4.5cm,yshift=-1.5cm]
\node (h1) at (0,0) {$(0,-1)$};
\end{scope}

\begin{scope}[node distance=0.3cm,on grid,align=center,font=\footnotesize,xshift=6cm,yshift=0cm]
\node (i1) at (0,0) {$(0,0)$};
\end{scope}

\begin{scope}[node distance=0.3cm,on grid,align=center,font=\footnotesize,xshift=7.5cm,yshift=1.5cm]
\node (j1) at (0,0) {$(1,-1)$};
\end{scope}

\begin{scope}[node distance=0.3cm,on grid,align=center,font=\footnotesize,xshift=7.5cm,yshift=0.5cm]
\node (k1) at (0,0) {$(-1,0)$};
\end{scope}

\begin{scope}[node distance=0.3cm,on grid,align=center,font=\footnotesize,xshift=7.5cm,yshift=-1.5cm]
\node (l1) at (0,0) {$(0,1)$};
\end{scope}

\begin{scope}[node distance=0.3cm,on grid,align=center,font=\footnotesize,xshift=9cm]
\node (m1) at (0,0) {$(0,0)$};
\end{scope}

\begin{scope}[node distance=0.3cm,on grid,align=center,font=\footnotesize,xshift=10.5cm,yshift=1.5cm]
\node (n1) at (0,0) {$(-1,1)$};
\end{scope}

\begin{scope}[node distance=0.3cm,on grid,align=center,font=\footnotesize,xshift=10.5cm,yshift=0.5cm]
\node (o1) at (0,0) {$(1,0)$};
\end{scope}

\begin{scope}[node distance=0.3cm,on grid,align=center,font=\footnotesize,xshift=10.5cm,yshift=-1.5cm]
\node (p1) at (0,0) {$(0,-1)$};
\end{scope}

\begin{scope}[node distance=0.3cm,on grid,align=center,font=\footnotesize,xshift=12cm]
\node (q1) at (0,0) {};
\end{scope}

\draw[->,semithick] (a1) to (b1);
\draw[->,semithick] (a1) to (c1);
\draw[->,semithick] (a1) to (d1);

\draw[->,semithick] (b1) to (e1);
\draw[->,semithick] (c1) to (e1);
\draw[->,semithick] (d1) to (e1);

\draw[->,semithick] (e1) to (f1);
\draw[->,semithick] (e1) to (g1);
\draw[->,semithick] (e1) to (h1);

\draw[->,semithick] (f1) to (i1);
\draw[->,semithick] (g1) to (i1);
\draw[->,semithick] (h1) to (i1);

\draw[->,semithick] (i1) to (j1);
\draw[->,semithick] (i1) to (k1);
\draw[->,semithick] (i1) to (l1);

\draw[->,semithick] (j1) to (m1);
\draw[->,semithick] (k1) to (m1);
\draw[->,semithick] (l1) to (m1);

\draw[->,semithick] (m1) to (n1);
\draw[->,semithick] (m1) to (o1);
\draw[->,semithick] (m1) to (p1);

\draw[->,semithick] (n1) to (q1);
\draw[->,semithick] (o1) to (q1);
\draw[->,semithick] (p1) to (q1);

\end{tikzpicture}}
\end{center}
\caption{The Auslander--Reiten quiver for the cluster category of type $D_4$, with bi-degrees corresponding to the grading bivector given in Section~\ref{ss:gradings-in-type-D}.}\label{fig:AR-quiver-D4-bigraded}
\vspace{1em}
\hrule
\end{figure}
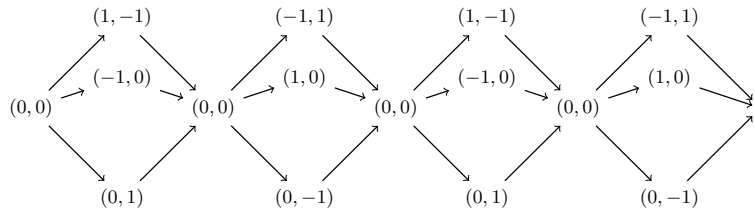

We have concentrated on the cluster category, for obvious reasons.  However we observe that the combinatorics described above can be extended to the bounded derived category $\curly{D}_{Q}$ associated to the category of finite dimensional $\complex Q$-modules.  By work of Happel (\cite{Happel}), the full subcategory of indecomposable objects of $\curly{D}_{Q}$ is equivalent to the mesh category of the repetition quiver $\integ Q$ (see \cite[Section~5]{Keller-CASurvey} for detailed definitions).  The Auslander--Reiten quiver of $\curly{D}_{Q}$ then takes the form of an infinite strip and in the case of type $A$, we may describe it as the quiver given by taking the finite strip corresponding to $\complex A_{n}$-modules repeated and reflected, as in Figure~\ref{fig:AR-quiver-D-A5} for type $A_{5}$.  

\begin{figure}
\begin{center} 
\scalebox{0.6}{\input{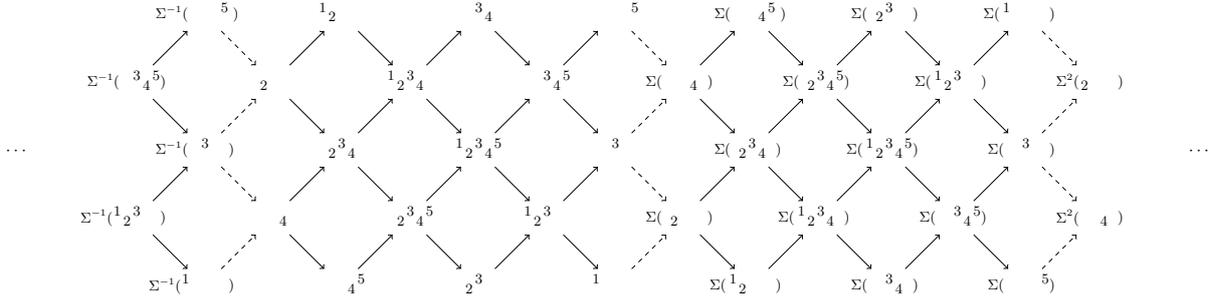}}
\end{center}
\caption{Part of the Auslander--Reiten quiver for the bounded derived category of type $A_5$.  (The morphisms going between the shifts of the $\complex A_{5}$-module category are indicated by dashed lines, to highlight the repetitive structure.)}\label{fig:AR-quiver-D-A5}
\vspace{1em}
\hrule
\end{figure}

Each repeated segment corresponds to an application of the suspension (or shift) functor of the mesh category, which we also denote by $\Sigma$.  The cluster category is constructed from the derived category by taking a certain quotient and inherits its own shift functor.  In the case at hand, this construction yields the M\"{o}bius strip of Figure~\ref{fig:AR-quiver-A5}.

We notice that if $n$ is odd, when we have the non-zero grading, the degree pattern above extends to give an exact tropical frieze pattern on the derived category---the shift functor reverses the parity of the degree, so it suffices to know the frieze pattern on $\curly{C}_{Q}$.  This is illustrated for $n=5$ in Figure~\ref{fig:AR-quiver-D-A5-with-degrees}.

\begin{figure}
\begin{center} 
\scalebox{0.6}{\input{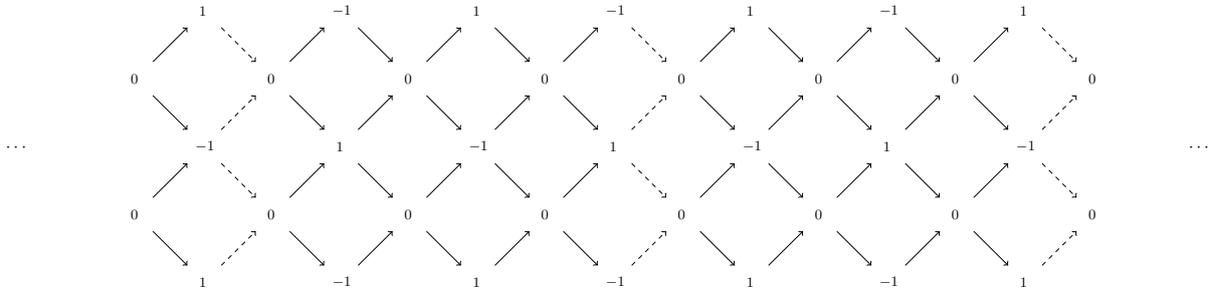}}
\end{center}
\caption{Part of the Auslander--Reiten quiver for the bounded derived category of type $A_5$, with degrees replacing modules.}\label{fig:AR-quiver-D-A5-with-degrees}
\vspace{1em}
\hrule
\end{figure}

We recall that we have no non-zero grading in the case of $A_{n}$ for $n$ even.  We observe that the derived category does admit a tropical frieze pattern similarly to the odd case, but this does not descend to the cluster category when $n$ is even: going from the derived category in type $A_{4}$ to the cluster category would see the objects $4$ and $\Sigma(\begin{tikzpicture}[baseline=0em] \begin{scope}[node distance=0.3cm,on grid,align=center,text width=1em,font=\footnotesize]
\node (mp1) at (0,0) {2};
\node (mp3) [above right=of mp1] {3};
\end{scope} \end{tikzpicture})$ being given degrees $-1$ and $0$, an inconsistency since these objects are isomorphic in the cluster category.

\section{Homogenisation}\label{s:homogenisation}

Given an initial seed $(\underline{x}=(X_{1},\dotsc ,X_{r}),B)$ and an arbitrary $r\cross d$ integer matrix $G$, we will not typically have a multi-graded seed.  However, we can make a modification to the initial data in order to homogenise this and produce a related cluster algebra of the same cluster algebra type that is multi-graded.  

\begin{lemma}\label{l:homog} Let $\curly{A}=\curly{A}(\underline{x},B)$ be a cluster algebra and let $G$ be an $r\cross d$ integer matrix.  Then there exists a multi-graded cluster algebra structure on a subalgebra $\curly{A}^{\text{hom}}$ of the field of rational functions $\mathbb{Q}(X_{1},\dotsc ,X_{r},h_{1},\dotsc ,h_{d})$ in which the elements $h_{i}$ are coefficients.  Furthermore the quotient cluster algebra obtained by setting $h_{i}=1$ for all $i$ is isomorphic (as a cluster algebra) to $\mathcal{A}$.
\end{lemma}

\begin{proof} We claim that the following is valid initial data for a multi-graded cluster algebra structure:
\begin{itemize}
\item $\underline{x}_{\text{hom}}=(X_{1},\dotsc ,X_{r},h_{1},\dotsc ,h_{d})$;
\item $B_{\text{hom}}=\left( \begin{smallmatrix} B \\ -G^{T}B \end{smallmatrix} \right)$;
\item $G_{\text{hom}}=\left( \begin{smallmatrix} G \\ I \end{smallmatrix} \right)$, with $I=I_{d\cross d}$ the identity matrix.
\end{itemize}
Here we are using the natural block matrix notation.

We have
\[ (B_{\text{hom}})^{T}G_{\text{hom}}=\left( \begin{smallmatrix} B^{T} & (-G^{T}B)^{T} \end{smallmatrix}\right)\left( \begin{smallmatrix} G \\ I \end{smallmatrix} \right)=B^{T}G-B^{T}G=0 \]
so that we have a multi-graded seed.  Hence we may construct $\curly{A}^{\text{hom}}$ within the field of rational functions $\mathbb{Q}(X_{1},\dotsc ,X_{r},h_{1},\dotsc ,h_{d})$.

It is clear that taking the quotient setting all $h_{i}$ to $1$ recovers $\curly{A}$.
\end{proof}

\begin{example} A simple example is as follows, observing that we could fix the lack of a grading in type $A_{n}$ for even $n$.  Let us take the quiver $A_{n}$, $n$ even, with a linear orientation
\begin{center} 
\scalebox{1}{\begin{tikzpicture}

\node (1) at (0,2) {1};
\node (2) [right=of 1] {2};
\node (3) [right=of 2] {3};
\node (4) [right=of 3] {4};
\node (dots) [right=of 4] {$\cdots$};
\node (n) [right=of dots] {$n$};

\draw[<-] (1) to (2);
\draw[<-] (2) to (3);
\draw[<-] (3) to (4);
\draw[<-] (4) to (dots);
\draw[<-] (dots) to (n);

\end{tikzpicture}}
\end{center}
We can easily check that this quiver admits no non-zero grading.  Let $G=(0\ 1\ 0\ 1\ \dotsm\ 0\ 1)^{T}$.  Following the above homogenisation procedure, we see that we should add one coefficient corresponding to an additional vertex $0$ to the quiver to give the ice quiver
\begin{center} 
\scalebox{1}{\begin{tikzpicture}

\node (0) [left=of 1,rectangle,draw=black] {0};
\node (1) at (0,2) {1};
\node (2) [right=of 1] {2};
\node (3) [right=of 2] {3};
\node (4) [right=of 3] {4};
\node (dots) [right=of 4] {$\cdots$};
\node (n) [right=of dots] {$n$};

\draw[<-] (0) to (1);
\draw[<-] (1) to (2);
\draw[<-] (2) to (3);
\draw[<-] (3) to (4);
\draw[<-] (4) to (dots);
\draw[<-] (dots) to (n);

\end{tikzpicture}}
\end{center}
This admits the grading $G_{\text{hom}}=(1\ 0\ 1\ 0\ 1\ \dotsm\ 0\ 1)^{T}$ (where we use the natural ordering, writing the degree at vertex $0$ first, rather than at the end as per the definition of $G_{\text{hom}}$).  We obtain a (graded) cluster algebra $\curly{A}'=\curly{A}(\underline{x}_{\text{hom}},B_{\text{hom}},G_{\text{hom}})$ of type $A_{n}$.
\end{example}

\begin{remark} In general, this does not yield a multi-graded cluster algebra structure on the polynomial extension $\curly{A}[h_{1},\dotsc ,h_{d}]$ of $\curly{A}$, since the new coefficients are involved in the exchange relations.  (Clearly it does if $G$ was in fact a multi-grading to begin with, for then we are simply adding $d$ degree $\underline{0}$ disconnected coefficients.  But then we have no need to homogenise.)
\end{remark}

There is an second, equally natural, way to homogenise a seed.  We may take \emph{principal coefficients}, that is, extend $B$ by the $n\cross n$ identity matrix (or indeed any diagonal matrix with diagonal entries $\pm 1$, as we prefer), adding one row per mutable variable.  Then we simply set the degree of the $i$th new variable to be whatever is required to homogenise at the $i$th position---namely the sum of the degrees over arrows leaving the $i$th vertex minus the sum of the degrees over arrows entering that vertex, if we are in the quiver setting.  This corrects inhomogeneity one mutable variable at a time, whereas the above lemma fixes inhomogeneity one coordinate of the multi-degree at a time.  Each construction might be appropriate in certain circumstances.

\begin{example} We extend the cluster algebra without coefficients of type $A_{2}$ by adding principal coefficients, so that the initial exchange quiver becomes
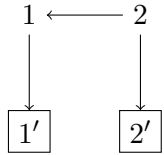
\begin{center}
\scalebox{1}{\begin{tikzpicture}

\node (1) at (0,2) {$1$};
\node (2) [right=of 1] {$2$};

\draw[<-] (1) to (2);

\node (a) [below=of 1,rectangle,draw=black] {$1'$};
\node (b) [below=of 2,rectangle,draw=black] {$2'$};

\draw[->] (1) to (a);
\draw[->] (2) to (b);

\end{tikzpicture}}
\end{center}
This quiver admits non-zero gradings, in contrast to type $A_{2}$ with no coefficients.  The space of solutions to $B_{k}G=0$ for mutable indices $k$ is 2-dimensional, with basis
\[ \{ G=(1, 0, 0, -1),\  H=(0, 1, 1, 0) \}. \]
\noindent The associated standard $\integ^{2}$-grading is represented by the following diagram:
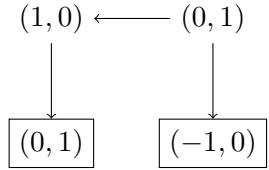
\begin{center}
\scalebox{1}{\begin{tikzpicture}

\node (1) at (0,2) {$(1,0)$};
\node (2) [right=of 1] {$(0,1)$};

\node (a) [below=of 1,rectangle,draw=black] {$(0,1)$};
\node (b) [below=of 2,rectangle,draw=black] {$(-1,0)$};

\draw[<-] (1) to (2);

\draw[->] (1) to (a);
\draw[->] (2) to (b);

\end{tikzpicture}}
\end{center}
Note that the frozen vertices $1'$ and $2'$ are exempted from the condition that the sums of the degrees at arrows entering and at arrows leaving the vertex are equal---we only require this at mutable vertices.

Let us take as our initial cluster $(X_{1},X_{2},X_{3},X_{4})$ (writing $X_{3}$ for $X_{1'}$ etc., for clarity).  Then the cluster variables and their degrees for the standard grading $(G,H)$, the two gradings $G$ and $H$ individually and their sum $G+H$ are as follows.

\begin{center}
\[ \begin{array}{c|ccccccc}
{} & X_{1} & X_{2} & X_{3} & X_{4} & \frac{X_{2}+X_{3}}{X_{1}} & \frac{X_{1}X_{4}+1}{X_{2}} & \frac{X_{2}+X_{3}+X_{1}X_{3}X_{4}}{X_{1}X_{2}}\rule[-0.75em]{0em}{0.75em} \\ \hline
(G,H) & (1,0) & (0,1) & (0,1) & (-1,0) & (-1,1) & (0,-1) & (-1,0) \\
G & 1 & 0 & 0 & -1 & -1 & 0 & -1 \\
H & 0 & 1 & 1 & 0 & 1 & -1 & 0 \\
G+H & 1 & 1 & 1 & -1 & 0 & -1 & -1
\end{array} \]
\end{center}

While we do again have the property of variables being concentrated in degrees $-1$, $0$ and $1$ for both gradings, as we had previously in type $A$, none of $(G,H)$, $G$ or $H$ is balanced.  There does exist a balanced grading however, namely $G+H$.
\end{example}

We conclude by observing that in both homogenisation constructions, i.e.\ that of Lemma~\ref{l:homog} and by taking principal coefficients, we may extend \emph{any} choice of $r\cross d$ matrix $G$ to a grading.  So once one fixes a coefficient pattern, the space of associated gradings is fixed but if one varies the coefficient pattern, one may obtain any grading one wants.  This justifies our earlier comment that properties of graded cluster algebras, e.g.\ classification, depend strongly on fixing a particular coefficient pattern, as opposed to being determined by the cluster algebra type alone.

\small

\bibliographystyle{amsplain}
\bibliography{biblio}\label{references}

\normalsize

\end{document}